\documentclass[11pt]{article}

\usepackage{amsmath,amsfonts,amssymb,amsthm,bbm}
\usepackage{graphics,epsfig}
\usepackage{color}
\usepackage{array}

\newtheorem{theorem}{Theorem}[section]
\newtheorem{proposition}{Proposition}[section]
\newtheorem{lemma}{Lemma}[section]
\newtheorem{corollary}{Corollary}[section]

\numberwithin{equation}{section}

\allowdisplaybreaks

\def\qed{\hfill$\Box$\medskip}

\begin{document}

\bigskip
\noindent{
\begin{center}
\Large\bf Tail asymptotic of the stationary distribution for the state dependent  (1,R)-reflecting random walk: near critical\footnote{The project is partially supported by the National
Natural Science Foundation of China (Grant No. 11131003). }
\end{center}
}

\noindent{
\begin{center}Wenming Hong\footnote{School of Mathematical Sciences
\& Laboratory of Mathematics and Complex Systems, Beijing Normal
University, Beijing 100875, P.R. China. Email: wmhong@bnu.edu.cn}  \ \ \
 Ke Zhou\footnote{ School of Mathematical Sciences
\& Laboratory of Mathematics and Complex Systems, Beijing Normal
University, Beijing 100875, P.R. China. Email:zhouke@mail.bnu.edu.cn}
\end{center}
}

\vspace{0.1 true cm}

\begin{center}
\begin{minipage}[c]{12cm}
\begin{center}\textbf{Abstract}\end{center}
\bigskip

In this paper, we consider the $(1,R)$ state-dependent reflecting random walk (RW) on the half line, allowing the size of jumps to the right at  maximal   $R$ and to the left only 1.
We provide an explicit criterion for  positive recurrence and the explicit expression of the stationary distribution based on the intrinsic branching structure within the walk.
As an application, we obtain the  tail asymptotic for the stationary distribution in the ``near critical" situation.

\mbox{}\textbf{Keywords:}\quad random walk, multi-type branching process,  positive recurrence, stationary distribution, tail asymptotic.\\
\mbox{}\textbf{Mathematics Subject Classification}:  Primary 60K37; Secondary 60J85
\end{minipage}
\end{center}


\section{ Introduction and Main Results\label{s1}}
\subsection{The background and motivation}
We consider the $(1,R)$-reflecting random walk on the half line,  i.e.,  a Markov chain $ \{X_m\}_{m\geq 0}$ on $\mathbb{Z^{+}}=\{0, 1, 2, \ldots \}$ with $X_0=0$ and the transition probabilities $P_{ij}$ specified by
 for $i\geq 0$ ($q(0)=0$),
\begin{equation*}
P_{ij} =
\begin{cases}
r(i), & \mbox{for $j=i$}, \\
q(i), & \mbox{for $j=i-1$}, \\
p_{j-i}(i), & \mbox{for $i<j\leq i+R$}, \\
0, & \mbox{otherwise,}
\end{cases}
\end{equation*}
 where $r(i)+q(i)+p_{1}(i)+p_{2}(i)+\cdots +p_{R}(i)=1$, $0<q(i)<1$, for $i\geq 1$, and $r(i)\geq 0, p_{1}(i), p_{2}(i), \cdots, p_{R}(i)\geq 0$. Obviously, this Markov chain  is irreducible. It can also be written as the transition matrix  (for simplicity, $R=2$),

\begin{equation*}
  \left(
             \begin{array}{ccccccc}
               r(0) & p_{1}(0) & p_{2}(0)\\
               q(1) & r(1) & p_{1}(1) & p_{2}(1)\\
               & q(2) & r(2) & p_{1}(2) & p_{2}(2)\\
               &  & q(3) & r(3) & p_{1}(3) & p_{2}(3)\\
               & & \ddots & \ddots & \ddots & \ddots & \\
             \end{array}
           \right).
\end{equation*}
in which all unspecified entries are zero.

For simplicity, we will restrict ourselves to consider $R=2$, and we write the transition probability at position $i$ as $P(i)=(q(i), r(i), p_{1}(i),p_{2}(i))$ (recall $q(0)=0$ and $0<q(i)<1$). At first, if the transition probability of the $(1,2)$-RW $(X_m)_{m\geq 0}$ is {\it state independent}, i.e.,~$ P(i)\equiv P=(q,r,p_{1},p_{2})$  for $i\geq 1$.
Let (see figure 1)
\begin{flalign*}
&D=\{(q,r,p_{1},p_{2}): p_{1}+p_{2}+q+r=1; ~~ p_{1}+2p_{2}< q\},\\
&L=\{(q, r,p_{1},p_{2}): p_{1}+p_{2}+q+r=1; ~~ p_{1}+2p_{2}= q\},
\end{flalign*}
it is easy to see that $(X_m)_{m\geq 0}$ is positive recurrent iff $ P(i)\equiv P=(q,r,p_{1},p_{2})\in D$ ($i\geq 1$) and null recurrent iff $ P(i)\equiv P=(q,r,p_{1},p_{2})\in L$ ($i\geq 1$).

How about the situation for the {\it state-dependent} $(1,R)$-RW $(X_m)_{m\geq 0}$ ? To our best knowledge only for  $R=1$, i.e., {\it state-dependent} $(1,1)$-RW,  the
criteria for the (positive) recurrence and the  expression for the
stationary distribution have been given explicitly (see for example
 ~\cite{K-M:57} and ~\cite{Kar}), and further tail asymptotic for the stationary distribution have been found in \cite{Durr} ( Page 294 and Page 305).

The aim of the present paper is to give an explicit criteria of the positive
recurrence and explicit expressions of  the stationary
distribution for the {\it state-dependent} $(1,R)$-RW,  which enable us to consider the tail asymptotic of the stationary distribution. Our method is based on the intrinsic branching structure within the random walk (\cite{hw}, \cite{hzh}).

\subsection{Main results}

\subsubsection{\it Criteria for the positive recurrence and stationary distribution}~
Define
\begin{flalign}\label{m}
& \alpha_{k}=p_{k}(0)+p_{k+1}(0)+\cdots+p_{R}(0),~\text{for}~~ 1\leq k\leq R,\nonumber\\
& \alpha=(\alpha_{1},\alpha_{2},\cdots,\alpha_{R}), ~  e_1=(1,0,\cdots,0),\nonumber\\
& \theta_{k}(i)=\frac{p_{k}(i)+p_{k+1}(i)+\cdots+ p_{R-1}(i)+p_{R}(i)}{q(i)},\nonumber\\
& {M_i} =\left(\begin{array}{ccccc}
\theta_{1}(i) & \theta_{2}(i) & \ldots & \theta_{R-1}(i) & \theta_{R}(i)\\
1 & 0 & \ldots & 0 & 0\\
\vdots & \vdots & \ddots & \vdots & \vdots\\
0 & 0 &\cdots & 0 & 0\\
0 & 0 & \cdots & 1 & 0
\end{array} \right)_{R\times R},
\end{flalign}


\begin{theorem}\label{sm}
 Assume  for $i\geq 0$,
\begin{equation}\label{at 0}
\begin{split}
&\mu(0)=1;~~~~\mu_{1}=\frac{1}{q(1)}\alpha e_{1}';\\
&\mu(i)=\frac{1}{q(i)}\alpha M_{1}M_{2}\cdots M_{i-1}e_{1}'.\end{split}\end{equation}

(i)~ $\mu(i)$ ($i\geq 0$) are the stationary measure of the {\it state-dependent} $(1,R)$-RW $(X_m)_{m\geq 0}$.

(ii) If $\sum_{i=0}^{\infty}\mu(i)<\infty$, then the walk $\{X_{m}\}_{m\geq 0}$ is positive recurrence. Furthermore the stationary  distribution can be expressed as \begin{equation}\label{sd}
\pi(i)=\frac{\mu(i)}{\sum_{i=0}^{\infty}\mu(i)}.\end{equation}
\end{theorem} \qed

\noindent{\bf Remark}~ (\ref{at 0}) generalize the classical results for the {\it state-dependent} $(1,1)$-RW, see for example \cite{Durr} (Page 297). \qed

\subsubsection{\it Tail asymptotic of the stationary distribution: near critical}~
With the explicit expression of the stationary distribution (\ref{sd}) at hand, we can consider the tail asymptotic of the  distribution. Firstly, it is not difficult (but is also not obviously, as   \cite{hzz} for the $(L,1)$-RW) to see that the tail of $\pi(i) $ is geometric decay in the sense $\lim_{i\rightarrow \infty}\frac{\log\pi({i})}{i}= -c<0$ when the transition probability $ P(i)\to P=(q,r,p_{1},p_{2})\in D$. What we are now interested in is the ``near critical" situation: the transition probability $ P(i)$ from the interior of the ``positive recurrence area $D$" to $ P $ in the ``null recurrence area $L$" as $i\to\infty$. See figure 1  (In this figure, we assume r=0).
\begin{center}
\includegraphics[totalheight=70mm]{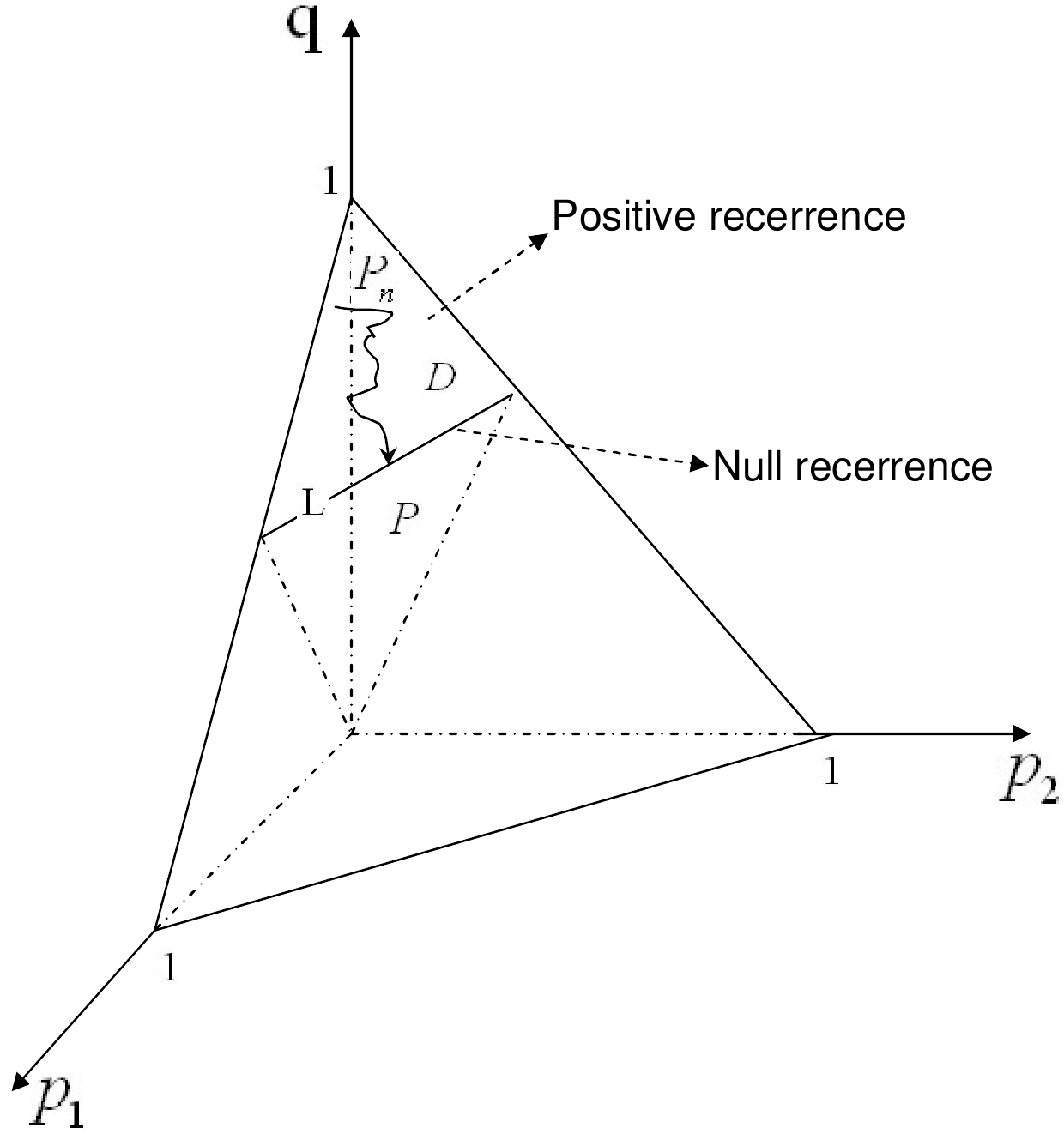}

figure 1: District of the transition probability
\end{center}
One of the interesting phenomena is that even all the $ P(i)\in D$, the ``positive recurrence area", the walk $X_{m}$ could be null recurrent.
To this end, we need to consider a finer manner of the $ P(i)$ goes to  $ P $ as $i\to\infty$.
Let $P=(p_{1},p_{2},r, q)\in L$,
 and for $i\geq 1$, $P(i)=(q(i), r(i), p_{1}(i),p_{2}(i))$ is given by
 \begin{equation}\label{tp}
\begin{split}
&p_{1}(i)= p_{1}-\varepsilon_{i}, ~p_{2}(i)= p_{2}-\varepsilon_{i}, ~q(i)= q+\varepsilon_{i};\\
&r(i)=1-p_{1}(i)-p_{2}(i)-q(i).
\end{split}
\end{equation}
    where  $\varepsilon_{i}>0$ and small enough, $\varepsilon_{i}\downarrow0$ as $i\rightarrow \infty.$  It is obvious that $P(i) \in D$, $P\in L$, and $P(i)\rightarrow P$ as $i\rightarrow \infty.$
\begin{theorem}\label{t1.3}
(a)~ If $\sum_{i=0}^{\infty}\varepsilon_{i}<\infty$,   $X_{m}$ is null recurrence.\\[8pt]
(b)~ if $\sum_{i=0}^{\infty}\varepsilon_{i}^{p}<\infty$ for some $1<p\leq2$,   $\kappa=\frac{4}{q}$.

(b1) When $\sum_{i=0}^{\infty}\prod_{k=0}^{i}e^{-\kappa\varepsilon_{k}}<\infty$, $X_{m}$ is positive recurrence, and
 \begin{equation*}\log\pi(i)\sim-\kappa\sum_{k=0}^{i}\varepsilon_{k},~~~\text{as} ~i\rightarrow \infty. \end{equation*}

(b2) When $\sum_{i=0}^{\infty}\prod_{k=0}^{i}e^{-\kappa\varepsilon_{k}}=\infty$, $X_{m}$ is null recurrence. \qed
\end{theorem}
As an application, we immediately have the following
\begin{corollary}\label{csd}
Suppose $\varepsilon_{i}\sim C i^{-\alpha}$ as $i\rightarrow \infty$, $C$ is a positive constant.\\
$Case~1:~\alpha>1$, $X_{m}$ is null recurrence.\\
$Case~2:~\frac{1}{2}<\alpha<1,$  $X_{m}$ is positive recurrence, and we have\begin{equation*}\log\pi(i)\sim-\frac{C\kappa}{1-\alpha}i^{1-\alpha},~~~\text{as} ~i\rightarrow \infty. \end{equation*}\\
$Case~3:~\alpha=1,$ If $C\kappa<1$, $X_{m}$ is null recurrence; if $C\kappa>1$, $X_{m}$ is positive recurrence, and\begin{equation*}\log\pi(i)\sim-C\kappa\log i,~~~\text{as} ~i\rightarrow \infty. \end{equation*}\qed
\end{corollary}

\noindent{\bf Remark}~ Theorem \ref{t1.3} and Corollary \ref{csd} say that even all the $ P(i)\in D$, the walk $X_{m}$ could be null recurrent, which generalize the results for the {\it state-dependent} $(1,1)$-RW (\cite{Durr}, Page 294 and Page 305). \qed

\

We arrange the remainder of this paper as follows. In Section \ref{s3}, we will prove Theorem \ref{sm} after a brief review about the intrinsic branching structure within the walk, which is the basic tool to specify the stationary measure;  Theorem \ref{t1.3} and Corollary \ref{csd} will be proved in Section \ref{s4}, together with some preparations on the asymptotic solution of difference system.



\section{ Proof of Theorem \ref{sm} \label{s3}}

Let  $
T=\inf\{m>0: X_{m}=0\}
$, $N(i)=\sum_{m=0}^{T-1}1_{\{X_{m}=i\}}$ be the number of
visits to state $i$ by the chain before $T$, and $E^i$ is the
expectation when the walk starts at $X_0=i$.
Firstly,   recall a classical results on the (positive)
recurrence and the stationary distribution of a general Markov
chain $X_m$.

\begin{proposition} (Thoerem (4.3), \cite{Durr}) \label{p1.2}
For $k\geq 0$, $\mu(i)=E^{0}N(i)$ defines a stationary measure. If $\sum_{k=0}^{\infty}\mu(i)<\infty$, the random walk is positive recurrence, and the stationary distribution can be expressed as \begin{equation*}
\pi(i)=\frac{\mu(i)}{\sum_{i=0}^{\infty}\mu(i)}.\end{equation*}\qed
\end{proposition}
We can calculate the $\mu(i)=E^{0}N(i)$ by the intrinsic branching structure within the $(1, R)$-RW as follows (the proof will delay at the end of this section),
\begin{proposition}\label{p1.1} We have $E^{0}N(0)=1$, and
\begin{equation*}
\begin{split}
E^{0}N(1)&=\frac{1}{q(i)}\alpha e_{1}'=\frac{p_{1}(0)+p_{2}(0)}{q(1)},\\
E^{0}N(i)&=\frac{1}{q(i)}\alpha M_{1}M_{2}\cdots M_{i-1}e_{1}',~~\text{for}~~ i>1,
\end{split}\end{equation*}
where $\alpha$, $M_i$ are given in (\ref{m}). \qed
\end{proposition}

\noindent {\it Proof of Theorem \ref{sm}.}~ With Proposition  \ref{p1.2}  and  \ref{p1.2}  at hand, Theorem \ref{sm} is immediately. \qed

\

What  we should do is to prove Proposition  \ref{p1.1}, our method is the intrinsic branching structure within the $(1, R)$-RW ~(\cite{hw}, 2009).

\

\noindent {\it Brief review for the intrinsic branching structure.}
The intrinsic branching structure within a random walk
{ has} been studied by many authors. For the $(1,1)$-RW,
Dwass~(\cite{D}, 1975) and Kesten {\textit{et al.}
(\cite{KKS}, 1975)} observed a Galton-Watson process with the geometric offspring distribution hidden in the
nearest random walk. The branching structure is a powerful tool in
the study of random walks in a random environment
(RWRE, for short). In ~\cite{KKS}, Kesten \textit{et
al.}, proved a stable law for the nearest RWRE by using this
branching structure. The key point is that the hitting time
 $T_{i}$ can be calculated accurately by the branching
structure.

However, if the random walk is allowed to  jump
even to a bounded range, referred to as the $(L,R)$-RW, the
situation will become much more complicated. A multi-type branching
process has been revealed by Hong \& Wang (\cite{hw}, 2009) for the $(L,1)$-RW, and a little bit late for the $(1,R)$-RW
(\cite{hzh}, 2010) by Hong \& Zhang. It must be emphasized that these two branching structures are not symmetric, instead
they are essentially different. Note that if
we assume $q_{2}(i)\equiv 0$, both branching structures degenerate
to the case of the $(1,1)$-RW.

The following discussion is based on $R=2$. The
general case can be similarly discussed, but the notation is much more
complicated. Assume that $X_{0}=0$ we can calculate $E^{0}N_{i}$ by
using the branching structure within the random walk~(\cite{hw},
2009). Note that we consider the reflected $(1,R)$-RW and  calculate $E^{0}N_{i}$ before first return the start position $0$, actually we use the branching structure for the $(R,1)$-RW
 by Hong \& Wang (\cite{hw}, 2009), and with a little modification because of considering the walk could be stay at each state $i$ (here $r(i)\geq 0$).

Recall that $T=\inf\{n>0,X_{n}=0\}$, define
\begin{flalign*}
    U^{m}_{k}&=\#\{0\leq j<T:X_{j}\leq k,~X_{j+1}=k+m\} \quad \text{
    for } k\geq 0, \; m=1,2.\\
    U^{3}_{k}&=\#\{0\leq j<T:~X_{j}=k,X_{j+1}=k\} \quad \text{
    for } k\geq 0,
\end{flalign*}
 Setting\begin{equation*}
    U_k=(U^{1}_{k},U^{2}_{k},U^{3}_{k}) \quad \text{ for } k\geq 0.
\end{equation*}
We then have the following property  {(with~a~little~ modification)}
\medskip

\noindent{\bf Theorem A} (Hong and Wang~\cite{hw})
{\it\noindent(1) The process $\{U_n\}_{n=0}^{\infty}$}
is a $3$-type branching process whose branching mechanism is given
by,
\begin{equation}\label{at 1}
\begin{split}
P(U_{0}&=(1,0,0))=p_{1}(0),\\
P(U_{0}&=(0,1,0))=p_{2}(0),\\
P(U_{0}&=(0,0,1))=r(0);
\end{split}\end{equation}
and  for $k\geq 0$
\begin{flalign*}
&P(U_{k+1}=(a,b,c)\big|U_{k}=e_1)=\frac{(a+b+c)!}{a!b!
c!}r(k)^{a}p_{1}(k)^{b}p_{2}(k)^{c}q(k),\\
&P(U_{k+1}=(a,1+b,c)\big|U_{k}=e_2)=\frac{(a+b+c)!}{a!b!
c!}r(k)^{a}p_{1}(k)^{b}p_{2}(k)^{c}q(k),\\
&P(U_{k+1}=(0,0,0)\big|U_{k}=e_3)=1.
\end{flalign*}

{\it (2) For the process $\{U_n\}_{n=0}^{\infty},$ let $\widetilde M_k$ be
the $3\times 3$ mean matrix whose $m$-th row is
$E(U_{k+1}|U_k=e_m),$ for $k\geq 0$. Then, one has that\begin{equation*}
\widetilde M_{k}= \left(
\begin{array}{ccc}
\frac{p_{1}(k)}{q(k)} & \frac{p_{2}(k)}{q(k)} & \frac{r(k)}{q(k)}\\
1+\frac{p_{1}(k)}{q(k)} & \frac{p_{2}(k)}{q(k)}& \frac{r(k)}{q(k)}\\
0 & 0 & 0
\end{array} \right), \quad k\geq 1.
\end{equation*}\qed
}
\medskip

Now we are at the position to prove Proposition \ref{p1.1}.

\noindent{\it Proof of Proposition \ref{p1.1} } It is not hard to
deduce the relationship between the random walk and the intrinsic
branching structure that   $E^{0}N(0)=1$, and for $i\geq 1,$
 $N(i)=U^{1}_{i-1}+|U_{i}|$ (where$|U_{i}|=U^{1}_{i}+U^{2}_{i}+U^{3}_{i}$).
\begin{equation*}
\begin{split}
E^{0}N(1)&=p_{1}(0)+E^{0}U_{0}\widetilde{M_1}(1,1,1)'\\
&=p_{1}(0)+E^{0}U_{0}(\frac{1}{q(1)}-1,\frac{1}{q(1)},0)'\\
&=p_{1}(0)+p_{1}(0)(\frac{1}{q(1)}-1)+p_{2}(0)\frac{1}{q(1)}\\
&=\frac{p_{1}(0)+p_{2}(0)}{q(1)}.
\end{split}
\end{equation*}For $i>1$, using the Markov property, we have
\begin{equation*}
  E^{0}(N(i)\big|U_{i-1},U_{i-2},...,U_0)=U^{1}_{i-1}+|U_{i-1} \widetilde{M_i}|.
\end{equation*}
As a consequence
\begin{equation*}
\begin{split}
  E^{0}N(i)&=E^{0}U_{i-1}e_{1}'+E^{0}U_{i-1}\widetilde{M_{i}}(1,1,1)'\\
  &=E^{0}U_{i-2}\widetilde{M}_{i-1}e_{1}'+E^{0}U_{i-2}\widetilde{M}_{i-1}\widetilde{M_{i}}(1,1,1)'\\
  &=E^{0}U_{0}\widetilde{M}_{1}\widetilde{M}_{2}\cdots\widetilde{M}_{i-1}e_{1}'+E^{0}U_{0}\widetilde{M}_{1}\widetilde{M}_{2}\cdots\widetilde{M}_{i-1}\widetilde{M_{i}}(1,1,1)'.
\end{split}
\end{equation*}
 By \eqref{at 1}, $E^{0}U_{0}=(p_{1}(0),p_{2}(0),r(0)):=\beta$,
\begin{equation}\label{bh}
\begin{split}
E^{0}N(i)&=\beta\widetilde{M}_{1}\widetilde{M}_{2}\cdots\widetilde{M}_{i-1}e_{1}'+\beta\widetilde{M}_{1}\widetilde{M}_{2}\cdots\widetilde{M}_{i-1}\widetilde{M_{i}}(1,1,1)'\\
&=\frac{1}{q(i)}\beta\widetilde{M}_{1}\widetilde{M}_{2}\cdots\widetilde{M}_{i-1}(1,1,0)'.
\end{split}
\end{equation}
Define\begin{equation*}
\widehat{M}_{i}=\left(
\begin{array}{ccc}
\frac{p_1(i)+p_2(i)}{q(i)} & \frac{p_2(i)}{q(i)} & \frac{r(i)}{q(i)}\\
1 & 0 & 0\\
0 & 0 & 0\\
\end{array} \right),
\end{equation*}
notice that\begin{equation*} \widetilde{M}_{k}= \left(
\begin{array}{ccc}
1 & 0 & 0\\
1 & 1 & 0\\
0 & 0 & 1
\end{array} \right)\cdot\widehat{M}_{k}\cdot\left(
\begin{array}{ccc}
1 & 0 & 0\\
1 & 1 & 0\\
0 & 0 & 1
\end{array} \right)^{-1}.
\end{equation*}
Substitute the above equation into ~\eqref{bh}, by some calculations
\begin{equation*}
\begin{split}
E^{0}N(i)&=\frac{1}{q(i)}\beta \left(
\begin{array}{ccc}
1 & 0 & 0\\
1 & 1 & 0\\
0 & 0 & 1
\end{array} \right) \widehat{M}_{1}\widehat{M}_{2}\cdots \widehat{M}_{i-1}\left(
\begin{array}{ccc}
1 & 0 & 0\\
1 & 1 & 0\\
0 & 0 & 1
\end{array} \right)^{-1}(1,1,0)'\\
&=\frac{1}{q(i)}\beta \left(
\begin{array}{ccc}
1 & 0 & 0\\
1 & 1 & 0\\
0 & 0 & 1
\end{array} \right)\left(\begin{array}{cc}M_{1} & 0\\ 0 & 0\end{array} \right)\left(\begin{array}{cc}M_{2} & 0\\ 0 & 0\end{array} \right)\cdots\left(\begin{array}{cc}M_{i-1} & 0\\ 0 & 0\end{array} \right)\left(
\begin{array}{ccc}
1 & 0 & 0\\
1 & 1 & 0\\
0 & 0 & 1
\end{array} \right)^{-1}(1,1,0)'\\
&=\frac{1}{q(i)}(p_{1}(0)+p_{2}(0),p_{2}(0)) M_{1}M_{2}\cdots
M_{i-1}(1,0)',
\end{split}
\end{equation*}
complete the proof. \qed

\section{ Proof of Theorem \ref{t1.3} \label{s4}}

\subsection{ The Asymptotic Solution of Difference system\label{s4}}
In this section, we introduce the asymptotic behavior of linear
difference system. Here, we just consider the second-order system.
\begin{equation}\label{f4.1} y_{k+1}=[\Lambda+R_{k}]y_{k} \quad
k\geq 0\end{equation}
 where $y_{k}\in \mathbb{R}^{2}$, $\Lambda=\text{diag}\{\lambda_{1},\lambda_{2}\}$, and $R_{n}$ is a small perturbation in a sense to be made precise. We assume that $|\lambda_{1}|>|\lambda_{2}|>0$, and $\|R_{n}\|=\sum_{i=1}^{n}\sum_{j=1}^{n}|r_{ij}|$.

The classical result for asymptotic analysis of solutions is to represent a fundamental matrix in the form
\begin{equation*}
Y_{k}=[I+o(1)]\prod_{i=0}^{k-1}\tilde{\Lambda}_{i}\end{equation*}
where $\tilde{\Lambda}(l)$ is an explicitly diagonal matrix whose main terms come from $\Lambda$.

If we consider the difference equations $
y_{k+1}=A_{k}y_{k}$. Use this asymptotic representation,    we can give a precise estimation of  the non-homogeneous matrix products $A_{n}A_{n-1}\cdots A_{k_{0}}$.

Here, we just give the case that $R_{n}$ in (\ref{f4.1}) is an $l^{1}$-perturbations and  $l^{p}$-perturbations with
$1<p\leq2$. If $p>2$, the result is more complicate.

\subsection{$l^{1}$-perturbations}
The fundamental theorem of Levinson (\cite{levin}, 1948) establish analogous results for perturbed systems of differential equations. Benzaid
and Lutz ~(\cite{BL}, 1987) give the discrete analogue for difference equations. This theorem consider the more general case when $\Lambda$ is depends on $k$, requiring a dichotomy condition on $\Lambda$ and a growth condition on the perturbation $R_{k}$.

\begin{proposition}(Theorem 2.2  Benzaid
and Lutz (P202))\label{pl1}
Consider $y_{k+1}=[\Lambda+R_{k}]y_{k}$, where $\Lambda=\text{diag}\{\lambda_{1},\lambda_{2}\cdots \lambda_{n}\}$, $\lambda_{i}\neq 0$ for all $1\leq i\leq n$, and $\sum_{k=k_{0}}^{\infty}\|R_{k}\|<\infty$. Then the system has a fundamental matrix satisfying, as $k\rightarrow\infty$
\begin{equation}\label{f4.2}
Y_{k}=[I+o(1)]\Lambda^{k}.
\end{equation}
\end{proposition}


\subsection{$l^{p}$-perturbations with
$1<p\leq 2$}
While the discrete version of Levinson's theorem considered $l^{1}$-perturbations R in
(\ref{f4.1}), the discrete version of the theorem of Hartman-Wintner ~(\cite{HW}, 1955)was concerned with $l^{p}$ perturbations
for some $1<p\leq 2$. The proof is based on the so-called $Q$-transformation which was first introduced
for differential equations by Harris and Lutz ~(\cite{HL}, 1974) and later on modified for difference equations by Benzaid and Lutz ~(\cite{BL}, 1987). Those methods have been well-established.

\begin{proposition}(Corollary 3.4  Benzaid
and Lutz (P210))\label{plp}
Consider $y_{k+1}=[\Lambda+R_{k}]y_{k}$, where $\Lambda=\text{diag}\{\lambda_{1},\lambda_{2}\cdots \lambda_{n}\}$, $|\lambda_{1}|>|\lambda_{2}|>\cdots |\lambda_{n}|>0$, and $\sum_{k_{0}}^{\infty}\|R_{k}\|^{p}<\infty$ for some $1<p\leq 2$. Then the system has a fundamental matrix satisfying, as $k\rightarrow\infty$
\end{proposition}
\begin{equation}\label{f4.3}
Y_{k}=[I+o(1)]\prod_{i=0}^{k-1}[\Lambda+\text{diag} R_{i}]
\end{equation}

\subsection{From ``positive
recurrence area" to the boundary of null recurrence   \label{s6}}
In this section, we just consider when $R=2$, and assume that $p_{2}(i),p_{2}>0$. The key to prove Theorem \ref{t1.3} is to discuss when $P(n)\rightarrow P$, the asymptotic representation of $M_{1}M_{2}\cdots M_{n}$. To this end, we consider the following  difference system.
\begin{equation}\label{f5.1}
y_{n+1}=M_{n}'y_{n}=[M'+R_{n}']y_{n}\end{equation}where $R_{n}=M_{n}-M,$
\begin{equation*}
 M= \left(
 \begin{array}{cc}
\frac{p_1+p_2}{q} & \frac{p_2}{q}\\
1 & 0
\end{array} \right),\quad
M_{n}= \left(
 \begin{array}{cc}
\frac{p_{1}(n)+p_{2}(n)}{q(n)} & \frac{p_{2}(n)}{q(n)}\\
1 & 0
\end{array} \right).
\end{equation*}




It is easy to see that $1=\lambda_{1}>0>\lambda_{2}>-1$, where $\lambda_{1},\lambda_{2}$ are two eigenvalues of  $M$. So it can be expressed in the diagonal form\begin{equation*}T^{-1}MT=\text{diag}(\lambda_{1},\lambda_{2})\end{equation*}
where $T$ is the non-singular matrix \begin{equation*} T= \left(
 \begin{array}{cc}
1 & \lambda_{2}\\
1 & 1
\end{array} \right).\end{equation*}
Let $y_{n}=(T^{-1})'z_{n}$, and (\ref{f5.1}) becomes
\begin{equation}\label{f5.2}
z_{n+1}=[\text{diag}(\lambda_{1},\lambda_{2})+T'R_{n}'(T^{-1})']z_{n}.\end{equation}
{{\begin{lemma}\label{l5.1}
Under our condition, we have for some constant $K_{1},K_{2}$, \begin{equation*}K_{1}\varepsilon_{n}\leq\|T'R_{n}'(T^{-1})'\|\leq K_{2}\varepsilon_{n}, ~~\text{as}~~ \varepsilon_{n}\rightarrow 0.\end{equation*}
\end{lemma}
\begin{proof}
For fixed $T$, by \cite{hor}, P295, Theorem~ 5.6.7. $\|T^{-1}\cdot T\|_{\diamond}=\|\cdot\|$ is also a norm of the matrix. Then by  the equivalence of the norm, there exists constants $c_{1},~c_{2}$,
\begin{equation*}c_{1}\|R_{n}\|\leq \|R_{n}\|_{\diamond} \leq c_{2}\|R_{n}\|.\end{equation*}
We can see
\begin{equation*}\|R_{n}\|=|\frac{p_{1}(n)+p_{2}(n)}{q(n)}-\frac{p_{1}+p_{2}}{q}|+|\frac{p_{2}(n)}{q(n)}-\frac{p_2}{q}|.\end{equation*}
Using $p_{1}(n)\sim p_{1}-C\varepsilon_{n}, ~p_{2}(n)\sim p_{2}-C\varepsilon_{n}, ~q(n)\sim q+C\varepsilon_{n}$, we have for some constant $C'$, $\|R_{n}\|\sim C'\varepsilon_{n}$. So there exist $K_{1},K_{2}$ satisfy \begin{equation*}K_{1}\varepsilon_{n}\leq\|T'R_{n}'(T^{-1})'\|\leq K_{2}\varepsilon_{n} \end{equation*}
\end{proof}
}}

Combine (\ref{f5.1}), (\ref{f5.2}), Lemma \ref{l5.1}  and the discuss above, we deduce from the Propositions \ref{pl1} and \ref{plp}  that,
\begin{lemma}\label{l5.2}
If $\sum_{n=0}^{\infty}\varepsilon_{n}<\infty$,
Then the system (\ref{f5.1}) has a fundamental matrix satisfying, as $k\rightarrow\infty$
\begin{equation}\label{f5.3}
Y_{k}=(T^{-1})'[I+o(1)]\text{diag} (1,\lambda_{2}^{k}).
\end{equation}
\end{lemma}

\begin{proof}
By Lemma \ref{l5.1}, if $\sum_{n=0}^{\infty}\varepsilon_{n}<\infty$, we use Proposition \ref{pl1} to system (\ref{f5.2}), we have
\begin{equation}
Z_{k}=[I+o(1)]\text{diag}(1,\lambda_{2}^{k}).
\end{equation}
Substitute $y_{n}=(T^{-1})'z_{n}$ we obtain (\ref{f5.3}),   complete the proof.
\end{proof}

\begin{lemma}\label{l5.3}
If for some $p$ such that $1<p\leq 2$, $\sum_{n=0}^{\infty}\varepsilon_{n}^{p}<\infty$.  Then the system (\ref{f5.1}) has a fundamental matrix satisfying, as $k\rightarrow\infty$
\begin{equation}\label{f4.2}
Y_{k}=(T^{-1})'[I+o(1)]\prod_{i=0}^{k-1}[\text{diag} T^{-1}M_{i}T]
\end{equation}
\end{lemma}
\begin{proof}
The proof is the same as Lemma 5.2,
\begin{equation*}\begin{split}
\text{diag}(\lambda_{1},\lambda_{2})+\text{diag}T'R_{n}'(T^{-1})'&=\text{diag}(\lambda_{1},\lambda_{2})+\text{diag}T'M_{n}'(T^{-1})'-\text{diag}T'M'(T^{-1})'\\
&=\text{diag}T'M_{n}'(T^{-1})'=\text{diag}T^{-1}M_{n}T
\end{split}
\end{equation*}
\end{proof}

\subsection{ Proof of Theorem \ref{t1.3}  \label{s7}}
 { {
\begin{proof} It is evident  that $\widetilde{Y}_i:=M_{i }'M_{i-1}'\cdots M_{1}'$ is a fundamental matrix of (\ref{f5.1}); on the other hand $Y_i$ in (\ref{f5.3}) and (\ref{f4.2}) is also the fundamental matrix of (\ref{f5.1}). So there exists a nonsingular matrix $\widetilde{C}$ such that $\widetilde{Y}_i=Y_{i}\cdot \widetilde{C}$.

(a) If $\sum_{n=0}^{\infty}\varepsilon_{n}<\infty$.

Combine Theorem
\ref{sm}, Lemma \ref{l5.2} and (\ref{f5.1}),
\begin{equation*}
\begin{split}
\mu(i)&=\frac{1}{q(i)}\alpha M_{1}M_{2}\cdots M_{i-1}e_{1}'=\frac{1}{q(i)}e_{1}M_{i-1}'M_{i-2}'\cdots M_{1}'\alpha' \\
&=\frac{1}{q(i)}e_{1}Y_{i-1}\widetilde{C}\alpha'\\
&=\frac{1}{q(i)}e_{1}(T^{-1})'[I+o(1)]\text{diag}(1,\lambda_{2}^{i-1})\widetilde{C}\alpha'.
\end{split}
\end{equation*}
Because $|\lambda_{2}|<1$, $\mu(i)$ has a positive limit as $i\rightarrow\infty$. So $\sum_{i=0}^{\infty}\mu(i)=\infty$, and $X_{n}$ is null recurrence (It is evident that $X_{n}$ is recurrent because all the transition probability $P_i$ is in the ``positive recurrent area $D$") .

(b) If $\sum_{n=0}^{\infty}\varepsilon_{n}=\infty,$ and for some $p$ such that $1<p\leq 2$, $\sum_{n=0}^{\infty}\varepsilon_{n}^{p}<\infty$. We can calculate $\text{diag} T^{-1}M_{i}T$,
\begin{equation*}
\begin{split}
\text{diag} T^{-1}M_{i}T
&= \frac{1}{1-\lambda_{2}}\left(
 \begin{array}{cc}
\frac{p_{1}(n)+2p_{2}(n)}{q(n)}-\lambda_{2} & 0\\
0 & \lambda_{2}(1-\frac{p_{1}(n)+2p_{2}(n)}{q(n)})-\frac{p_{2}(n)}{q(n)}
\end{array} \right)\\
&\sim\frac{1}{1-\lambda_{2}}\left(
 \begin{array}{cc}
1-\kappa\varepsilon_{n}-\lambda_{2} & 0\\
0 & -\lambda_{2}\kappa\varepsilon_{n}-\frac{p_{2}(n)}{q(n)}
\end{array} \right).
\end{split}
\end{equation*}
By Lemma \ref{l5.3}, use the same method, there exist an constant $c$, such that
\begin{equation*}
\mu(i)\sim\frac{c}{(1-\lambda_{2})q(i)}\prod_{k=0}^{i-1}(1-\kappa\varepsilon_{k}).
\end{equation*}
Using the fact that $\text{log}(1-\alpha)\sim -\alpha$ as $\alpha\rightarrow 0 $ we see that
\begin{equation*}
\log\prod_{k=0}^{i-1}(1-\kappa\varepsilon_{k})\sim -\sum_{k=0}^{i-1}\kappa\varepsilon_{i}~~\text{as}~k\rightarrow \infty,
\end{equation*}then
\begin{equation}\label{f5.7}
\mu(i)\sim\frac{c}{(1-\lambda_{2})q(i)}e^{-\sum_{k=0}^{i-1}\kappa\varepsilon_{k}}.
\end{equation}
Note that $\frac{c}{(1-\lambda_{2})q(i)}$ is a bounded sequence.  If $\sum_{i=0}^{\infty}\prod_{k=0}^{i}e^{-\kappa\varepsilon_{k}}<\infty$ ,by (\ref{f5.7}), $\sum_{i=0}^{\infty}\mu(i)<\infty$,   the process is positive recurrence. If $\sum_{i=0}^{\infty}\prod_{k=0}^{i}e^{-\kappa\varepsilon_{k}}=\infty$, by (\ref{f5.7}), $\sum_{i=0}^{\infty}\mu(i)=\infty$,   it is null recurrence.
when the process is positive recurrence, \begin{equation}\log\pi(i)=\log\frac{\mu(i)}{\sum_{i=0}^{\infty}\mu(i)}
\sim-\kappa\sum_{k=0}^{i}\varepsilon_{k}. \end{equation}
\end{proof}
}}

\subsection{ Proof of Corollary \ref{csd}   \label{s8}}
In this section, we consider the example to explain the boundary between null
 recurrence and positive recurrence. For the case   $R=1$, it can be found in (\cite{Durr}).

Recall that $\kappa=\frac{4}{q}$. We assume $p_{1}(0)=p_{1}-Cn^{-\alpha}, ~p_{2}(0)=p_{2}-Cn^{-\alpha}$, and for $n>0$, $p_{1}(n)=p_{1}-Cn^{-\alpha}, ~p_{2}(n)=p_{2}-Cn^{-\alpha}, ~q(n)=q+Cn^{-\alpha}$, where $p_{1}, p_{2}, q$ satisfy
$p_{1}+2p_{2}=q.$
$Case~1:~\alpha>1.$ We have  $\sum_{n=0}^{\infty}\varepsilon_{n}<\infty$. By Theorem \ref{t1.3} $(a)$, $X_{n}$ is null recurrence.\\
$Case~2:~\frac{1}{2}<\alpha<1.$ There exist $1<p\leq 2$ such that $\sum_{n=0}^{\infty}\varepsilon_{n}^{p}<\infty$. And \begin{equation*}\prod_{k=0}^{i}e^{-\kappa\varepsilon_{k}}\sim e^{-\frac{C\kappa}{1-\alpha}i^{1-\alpha}}.\end{equation*} So $\sum_{i=0}^{\infty}\prod_{k=0}^{i}e^{-\kappa\varepsilon_{k}}<\infty$. By Theorem \ref{t1.3} $(b)$, the process is positive recurrence, and we have\begin{equation*}\log\pi(i)\sim-\frac{C\kappa}{1-\alpha}i^{1-\alpha},~~~\text{as} ~i\rightarrow \infty. \end{equation*}\\
$Case~3:~\alpha=1.$ It is easy to see that $\sum_{n=0}^{\infty}\varepsilon_{n}^{2}<\infty$. And\begin{equation*}\prod_{k=0}^{i}e^{-\kappa\varepsilon_{k}}\sim e^{-C\kappa\log i}=i^{-C\kappa}.\end{equation*} So if $C\kappa<1$, $\sum_{i=0}^{\infty}\prod_{k=0}^{i}e^{-\kappa\varepsilon_{k}}=\infty$ $X_{n}$ is null recurrence; if $C\kappa>1$, $\sum_{i=0}^{\infty}\prod_{k=0}^{i}e^{-\kappa\varepsilon_{k}}<\infty$ , $X_{n}$ is positive recurrence, and\begin{equation*}\log\pi(i)\sim-C\kappa\log i,~~~\text{as} ~i\rightarrow \infty. \end{equation*} \qed


\end{document}